\documentclass{amsart}
\usepackage{amsmath, amssymb}
\begin{document}

\newtheorem{theorem}{Theorem}[section]
\newtheorem{lemma}[theorem]{Lemma}
\newtheorem{corollary}[theorem]{Corollary}
\newtheorem{fact}[theorem]{Fact}
\newtheorem{proposition}[theorem]{Proposition}
\newtheorem{claim}[theorem]{Claim}

\theoremstyle{definition}
\newtheorem{example}[theorem]{Example}
\newtheorem{remark}[theorem]{Remark}
\newtheorem{definition}[theorem]{Definition}
\newtheorem{question}[theorem]{Question}

\def\id{\operatorname{id}}
\def\orb{\operatorname{orbit}}
\def\PP{\mathbb{P}}
\def\cb{\operatorname{Cb}}
\def\tp{\operatorname{tp}}
\def\stp{\operatorname{stp}}
\def\acl{\operatorname{acl}}
\def\acldim{\operatorname{acl-dim}}
\def\dcl{\operatorname{dcl}}
\def\eq{\operatorname{eq}}
\def\th{\operatorname{Th}}
\def\locus{\operatorname{loc}}
\def\aut{\operatorname{Aut}}
\def\gc{\operatorname{GC}}
\def\corr{\operatorname{Corr}}
\def\douady{\operatorname{Douady}}
\def\loc{\operatorname{loc}}

\title{$\aleph_0$-categorical  strongly minimal compact complex manifolds}
\author{Rahim Moosa}
\thanks{Rahim Moosa was partially supported by an NSERC Discovery Grant}
\address{Department of Pure Mathematics\\
University of Waterloo\\
Waterloo\\ Ontario N2L 3G1\\
Canada}
\author{Anand Pillay}
\thanks{Anand Pillay was partially supported by EPSRC grant  EP/F009712/1, a Marie Curie Chair, as well as the 
Humboldt Foundation. He would also like to thank Daniel Huybrechts for some helpful conversations during a visit to 
Bonn in April 2007}
\address{School of Mathematics\\
University of Leeds\\
Leeds LS2 9JT\\
UK}

\date{July 5th, 2010}


\begin{abstract}
{\em Essential} $\aleph_0$-categoricity; i.e., $\aleph_0$-categoricity in some full countable language, is shown to be a robust notion for strongly minimal compact complex manifolds.
Characterisations of triviality and essential $\aleph_0$-categoricity are given in terms of complex-analytic automorphisms, in the simply connected case, and correspondences in general.
As a consequence it is pointed out that an example of McMullen yields a strongly minimal compact K\"ahler manifold with trivial geometry but which is not $\aleph_{0}$-categorical, giving a counterexample to a conjecture of the second author and Tom Scanlon.
\end{abstract}

\maketitle

\section{Introduction and preliminaries}
This paper is concerned with a model-theoretic study of compact complex manifolds ($ccm$'s) $X$ which have ``little structure", in the sense of there 
being ``few" subvarieties of $X^{n}$ for all $n$.  Among the motivations for writing the current paper is to point out the existence of a strongly 
minimal compact K\"ahler manifold with trivial geometry in the model-theoretic sense, but which is not $\aleph_{0}$-categorical. This 
provides a counterexample to a conjecture of the second author and Tom Scanlon (analogous to a similar conjecture for strongly minimal differential 
algebraic varieties) that appears in~\cite{pillayscanlon2001}.
However, a closer look reveals that the notion of $\aleph_{0}$-categoricity itself is not so clear-cut for $ccm$'s, and so a large part of the current 
paper is dedicated to showing that at least for strongly minimal $ccm$'s, $\aleph_{0}$-categoricity is a robust notion, and in fact can be characterized by the existence of only finitely many {\em correspondences}: proper complex-analytic subsets of $X\times X$ that project onto $X$ in each co-ordinate.

A compact complex manifold $X$ can 
be considered as a first-order structure $\mathcal A(X)$ by adjoining predicates for all (closed) complex-analytic subsets of $X^{n}$ for all $n$. The first-order theory of the corresponding structure is very tractable from the model-theoretic point of view; it has finite Morley rank. There is a considerable (geometric) model-theoretic machinery around first-order theories of finite Morley rank. In so far as compact complex manifolds $X$ are concerned, the relevance of this model-theoretic machinery is inversely proportional to the extent to which $X$ is an algebraic variety. Loosely speaking, the strongly minimal compact complex manifolds, which are exactly the irreducible $ccm$'s with no proper infinite complex-analytic subsets, are the building blocks of arbitrary $ccm$'s.
There is a rudimentary classification of strongly minimal structures $M$ according to the behaviour of algebraic closure in a saturated elementary extension; (a) nonmodular, (b) modular nontrivial, and (c) trivial. When $M = 
\mathcal A(X)$, this essentially corresponds to (a) $X$ is an algebraic curve, (b) $X$ is a nonalgebraic simple complex torus, and (c) $X$ has algebraic and Kummer dimension zero, or equivalently $X$ admits no positive-dimensional  compact complex-analytic family of correspondences. 
The identification of strongly minimal $ccm$'s $X$ of type~(c) would seem to be a central problem in bimeromorphic geometry.
A further distinction within type~(c) is between $\aleph_{0}$-categorical and non $\aleph_{0}$-categorical.
But the notion is problematic. If $M$ is a structure for a {\em countable} language then $M$ (or rather the first-order theory of $M$) is said to be {\em $\aleph_{0}$-categorical} if $\th(M)$ has a unique countable model, equivalently if for each $n$, there are only finitely many $\emptyset$-definable subsets of $M^{n}$.
However the underlying first-order language of the structure $\mathcal A(X)$ is on the face of it {\em uncountable}, as for example each point of $X$ is
named by a predicate.
A possible definition of $\aleph_{0}$-categoricity of $X$ is that there is {\em some} full countable language $\mathcal{L}$ for $X$ (see Definition 1.1 below) such that $(X,\mathcal{L})$ is $\aleph_{0}$-categorical. Now $X$ need not have a full countable language, for example if $X$ is a Hopf surface. On the other hand, if $X$ does have a full countable language (which is the case when $X$ is of K\"ahler type), then as the first author points out in ~\cite{sat}, there is a ``canonical" choice for such a language, the so-called Douady language. We prove (Theorem 3.15) that for $X$ strongly minimal $X$ is $\aleph_{0}$-categorical in some full countable language iff $X$ is $\aleph_{0}$-categorical in the Douady language iff $X$ has only finitely many correspondences. In section 2 we go through the special case when $X$ is simply connected, where the arguments are easier and where the third condition becomes $\aut(X)$ is finite. McMullen's example of a general $K3$ surface $X$ with $\aut(X) = \mathbb{Z}$, provides then a trivial strongly minimal K\"ahler manifold which is not $\aleph_{0}$-categorical.

There are several overviews of the model theory of compact complex manifolds for a general audience, such as ~\cite{amsmoosa} and ~\cite{moosapillay}, which we point the reader towards. An introduction to key notions of model theory, especially in regard to applications, appears in ~\cite{pillaylms} which we again recommend for the non expert.
Douady spaces and full countable languages play an important role in the current paper, and the reader is referred to the first author's paper~\cite{sat}  for a comprehensive treatment.

\smallskip

As strong minimality and triviality are central to the paper we give brief accounts.
First, when we speak of a {\em definable set} in a structure $M$ we mean a set (typically a subset of some $M^{n}$) definable in $M$ possibly with parameters from $M$.
A one-sorted structure $M$ (in a possibly uncountable language $L$) is said to be {\em strongly minimal} if for any elementary extension $M'$ of $M$ every definable subset of $M'$ is finite or cofinite. 
If $M$ is strongly minimal and $M'$ a saturated elementary extension of $M$, algebraic closure yields an infinite-dimensional pregeometry or matroid on $M'$.
The structure $M$ is said to have {\em trivial geometry} if this pregeometry on $M'$ is trivial, in the sense that for any subset $A$ of $M'$, 
$\acl(A)\cap M = \cup_{a\in A}\acl(a)\cap M$.

As mentioned above, if $X$ is a compact complex variety (so reduced and irreducible), by $\mathcal{A}(X)$ we mean the structure which has $X$ as its universe and a predicate for each complex-analytic subset of each finite cartesian power of $X$.
We say that {\em $X$ is strongly minimal} or that {\em $X$ has trivial geometry}, if it is true of the structure $\mathcal A(X)$.
It follows from quantifier elimination that $X$ is strongly minimal just if $X$ has no positive-dimensional proper complex-analytic subsets.
We denote by $\mathcal A$ the many-sorted structure where there is a sort for each compact complex variety and a predicate for each complex-analytic subset of each finite cartesian product of sorts.
We typically work in the many-sorted structure $\mathcal A$. But note that 
by definability of types, a subset of $X^n$ is definable in $\mathcal A(X)$ if and only if it is definable in~$\mathcal A$.

\begin{definition}[cf.~\cite{sat}]
\label{fcl}
Suppose $X$ is a compact complex variety.
A {\em full countable language for $X$} is a countable (one-sorted, relational) language $\mathcal{L}$ and an $\mathcal{L}$-structure on $X$ such that
\begin{itemize}
\item[(1)]
for all $n<\omega$, a subset of $X^n$ is definable in $\mathcal{A}(X)$ if and only if it is definable (with parameters) in $(X,\mathcal{L})$.
\end{itemize}
We will say that $\mathcal{L}$ is {\em analytic} if in addition
\begin{itemize}
\item[(2)]
whenever $\sigma$ is an automorphism of $(X,\mathcal{L})$ and $A\subseteq X^n$ is a complex-analytic subset, then $\sigma(A)$ is complex-analytic.
\end{itemize}
We also say that $X$ is {\em essentially saturated} if it has some full countable language.
\end{definition}

\begin{example}[The Douady Language]
Suppose $X$ is an essentially saturated compact complex variety.
In~\cite{sat} it is shown that for all $n>0$, every irreducible complex-analytic subset of $X^n$ lives in a compact component of the Douady space $D(X^n)$.
(In fact this characterises essential saturation.)
Let us call the components of $D(X^n)$ that arise in this way {\em prime}.
Consider the reduct $\mathcal{A}_X$ of $\mathcal{A}$, where there is  a predicate for the restriction of the universal family $Z(X^n)\to D(X^n)$ to each prime component, as $n>0$ varies.
Then, by quantifier elimination, a subset of a cartesian power of $X$ is definable in $\mathcal A_X$ if and only if it is definable in $\mathcal A(X)$.
Now let $\mathcal{L}_{\douady}$ be the language where there is a predicate for each subset of $X^n$ that is $0$-definable in $\mathcal{A}_X$, for each $n>0$.
By definability of types this is a full countable language for $X$, and every automorphism of $(X,\mathcal{L}_{\douady})$ extends to an automorphism of $\mathcal{A}_X$.
Since automorphisms of $\mathcal A_X$ preserve complex-analyticity, $\mathcal{L}_{\douady}$ is a full countable analytic language for $X$. We call it the {\em Douady language} of $X$.
Note also that $(X,\mathcal{L}_{\douady})$ admits quantifier elimination.
\end{example}

Throughout, by $\acl$ we mean the algebraic closure in ``eq".

\begin{lemma}
\label{qefcl}
If $\mathcal{L}$ is a full countable analytic language for $X$ and $C\subseteq X^n$ is $F$-definable in $(X,\mathcal{L})$, then $C$ is of the form $\displaystyle \bigcup_{i=1}^\ell A_i\setminus B_i$ where each $A_i$ is an irreducible $\acl(F)$-definable complex-analytic subset of $X^n$ and $B_i$ is a proper $\acl(F)$-definable complex-analytic subset of $B_i$.
\end{lemma}

\begin{proof}
Let $G$ be the group of automorphism of $(X,\mathcal L)$, in the model-theoretic sense, that fixes $F$ point-wise.
Suppose $D$ is any $F$-definable set.
Since $\mathcal{L}$ is analytic, every member of $G$ will permute the collection of complex-analytic subsets of $X^n$ that contain $D$.
Hence the closure of $D$, $\bar D$, which we know is definable in $(X,\mathcal L)$ by fullness, is fixed set-wise by every automorphism in $G$.
By saturation, it follows that $\bar D$ is $F$-definable in $(X,\mathcal L)$.

Now, by quantifier elimination in $\mathcal{A}(X)$ we can write $C$ irredundantly as $C= \bigcup_{i=1}^\ell C_i$ where $C_i=A_i\setminus B_i$ with $A_i$ irreducible complex-analytic and $B_i$ a proper complex-analytic subset of $A_i$.
Moreover this decomposition is unique up to a permutation of $\{C_1,\dots,C_\ell\}$.
Since $\mathcal L$ is analytic and $C$ is $F$-definable, it follows that every member of $G$ permutes $\{C_1,\dots,C_\ell\}$.
Hence, by saturation, each $C_i$ is $\acl(F)$-definable in $(X,\mathcal L)$.
By the discussion in the first paragraph, applied to $D=C_i$ and the parameter set $\acl(F)$, $A_i=\bar C_i$ is also $\acl(F)$-definable.
Hence, so is $B_i=A_i\setminus C_i$.
\end{proof}

\begin{definition}[Essential $\aleph_0$-categoricity]
A compact complex variety $X$ is 
{\em essentially $\aleph_0$-categorical} if  there exists a full countable language $\mathcal L$ for $X$ such that $(X,\mathcal L)$ is $\aleph_0$-categorical.
\end{definition}

This paper is primarily concerned with essentially $\aleph_0$-categorical strongly minimal manifolds.
If $X$ is such, and $\mathcal L$ is a full countable language witnessing $\aleph_0$-categoricity, then by Zilber's theorem algebraic closure in $(X,\mathcal L)$, and hence also in $\th\big(\mathcal A(X)\big)$, is a modular geometry.
But modular strongly minimal manifolds are characterised in Proposition~5.1 of~\cite{pillayscanlon2003}, they are either of trivial geometry or are simple complex tori.
Since the latter are not essentially $\aleph_0$-categorical (by what we know about $\aleph_0$-categorical groups, for example), we obtain:

\begin{fact}
\label{cat-triv}
Essentially $\aleph_0$-categorical strongly minimal compact complex varieties are necessarily of trivial geometry.
\end{fact}

The following is a useful characterisation of triviality.

\begin{lemma}
\label{trivial-condition}
Suppose $X$ is a strongly minimal compact complex variety. 
The following are equivalent:
\begin{itemize}
\item[(i)]
$X$ has trivial geometry,
\item[(ii)]
there is no infinite definable family of irreducible complex-analytic subsets of $X^2$ projecting onto $X$ in each co-ordinate.
\end{itemize}
\end{lemma}
\begin{proof} This is well-known and comes easily out of the definitions, but we give a sketch of the proof anyway, using freely model-theoretic language.
Let us first assume that there is some infinite definable family of irreducible complex-analytic subsets of $X^{2}$ projecting onto $X$ in each coordinate.
Namely there is some definable subset $W$ of $X^{k}$ (some $k$), and some definable subset $Z$ of $W\times X^{2}$ such that for each $w\in W$ the fibre  $Z_w\subset X^{2}$ is an irreducible complex-analytic set that projects onto $X$ in each coordinate, and $\{Z_w:w\in W\}$ is infinite.
As we may exclude $X^{2}$ itself from being among the fibres, by strong minimality each $Z_w$ is generically finite-to-one over $X$ in each co-ordinate.
In particular each $Z_w$ is itself a strongly minimal compact complex variety.

Now we pass to a saturated elementary extension $\mathcal{A}'$ of $\mathcal{A}$. Let $(c,a,b)\in Z(\mathcal A')$ be a generic point of $Z$ in $\mathcal A'$.
As $Z(\mathcal A')_c$ is generically finite-to-one over $X(\mathcal A')$ in each co-ordinate, $a$ and $b$ are interalgebraic over $c$.
On the other hand, as $\acldim(Z)>\acldim(W)$, $a,b\notin\acl(c)$.

We claim that $b\notin\acl(a)$.
Indeed, let $d\models\tp(c)$ be indepenendent over $\{a,c\}$.
As $\{Z_w:w\in W\}$ is infinite, independence from $c$ implies that $Z(\mathcal A')_c\neq Z(\mathcal A')_d$ and so $Z(\mathcal A')_c\cap Z(\mathcal A')_d$ is a finite set.
As $a\notin\acl(c,d)$, it follows that $(a,b)\notin Z(\mathcal A')_c\cap Z(\mathcal A')_d$.
But if $b$ were contained in $\acl(a)$ then $c$ and $d$ would be independent over $\{a,b\}$, and stationarity would imply that $\tp(c/ab)=\tp(d/ab)$, which would in turn imply the contradiction $(a,b)\in Z(\mathcal A')_c\cap Z(\mathcal A')_d$.

So
$b\in\acl(c,a)\setminus\big(\acl(c)\cup\acl(a)\big)$, showing nontriviality of $X$. 

Conversely assume $X$ is nontrivial.
Let $a,b\in X(\mathcal A')$ and $c\in X(\mathcal A')^k$ be such that $b\in\acl(c,a)\setminus\big(\acl(c)\cup\acl(a)\big)$.
Extending $c$ we may further assume that $\tp(ab/c)$ is stationary.
Let $Z\subset X^k\times X^2$ be the locus of $(c,a,b)$.
Then for some definable $W\subseteq X^n$ with $c\in W(\mathcal A')$, each $Z_w\subseteq X^2$ is an irreducible complex-analytic subset that projects onto $X$ in each co-ordinate.
Note that if $E$ is the definable equivalence relation on $W$ where $wEw'$ if and only if $Z_w=Z_w'$, then $Z(\mathcal A')_c$ is defined over $c/E$ and so $b\in\acl(c/E,a)$.
Since $b\notin\acl(a)$, it follows that $\{Z_w:w\in W\}$ must be infinite.
\end{proof}

\section{A warm-up: the simply connected case} 
\label{sc-section}

\begin{lemma}
\label{scsm}
Suppose $X$ is a simply connected strongly minimal compact complex manifold of dimension greater than one.
\begin{itemize}
\item[(a)]
The only irreducible complex-analytic subsets of $X^2$ are points, $X^2$ itself, vertical and horizontal ``slices'' $\{a\}\times X$ and $X\times \{a\}$ where $a\in X$, and graphs of automorphisms.
\item[(b)]
Suppose moreover that $X$ has trivial geometry, and $A\subseteq X^n$ is an irreducible complex-analytic subset.
Then, after some permutation of the co-ordinates, there exists $0\leq r\leq n$ such that $A$ is defined by equations $\phi_1,\dots,\phi_{n-r}$ where $\phi_j$ is either of the form `$x_{r+j}=\sigma(x_i)$' for some $\sigma\in \aut X$ and $i\leq r$, or of the form `$x_{r+j}=b$' for some $b\in X$.
\end{itemize}
\end{lemma}

\begin{proof}
To prove part~(a) we first recall the following fact:

\begin{fact}
\label{norm}
Suppose $X$ is a strongly minimal compact complex manifold of dimension greater than one, $A\subseteq X^n$ is an irreducible complex-analytic subset such that one of the co-ordinate projections $\pi:A\to X$ is surjective and finite-to-one.
If $\rho:A'\to A$ is a normalisation of $A$, then $\pi\circ\rho:A'\to X$ is an unramified covering.
\end{fact}

\noindent
Indeed, arguments for this fact can be found in Lemma~7 of~\cite{pillay-torus}, Proposition~2.12 of~\cite{pillayscanlon2001}, and Lemma~4.2 of~\cite{moosamorarutoma}.
Here is a sketch:
The branch locus of $\pi\circ\rho:A'\to X$, where the morphism is not locally a biholomorphism, is a proper complex-analytic subset of $A'$.
Since $\pi\circ\rho$ is finite-to-one and $A'$ is irreducible, the image of the branch locus cannot be all of $X$, and so by strong minimality it must be a finite subset of $X$.
Hence the branch locus itself is finite.
But the smoothness of $X$ and the normality of $A'$ imply that the branch locus is either of codimension one or empty (by the purity of branch theorem).
It follows that $\pi\circ\rho$ is everywhere locally biholomorphic, as desired.

Now, if $A\subseteq X^2$ is an irreducible complex-analytic set that is neither all of $X^2$ nor a ``slice'', then by strong minimality each co-ordinate projection $\pi:A\to X$ is finite-to-one onto $X$.
Applying Fact~\ref{norm} to this situation, and remembering that $X$ is simply connected, we get that $\pi\circ\rho$, and hence $\pi:A\to X$ itself, is a biholomorphism.
This proves part~(a).

We prove part~(b) by induction on $n$.
The case of $n=1$ is by strong minimality, and the case of $n=2$ is part~(a).
For the induction step, suppose $n>2$ and let $a=(a_1,\dots,a_n)\in X(\mathcal{A}')$ be a generic point of $A$ in a sufficiently saturated elementary extension $\mathcal A'$ of $\mathcal A$.
After permuting co-ordinates we may assume that $\{a_1,\dots,a_r\}$ is an $\acl$-basis for $\{a_1,\dots,a_n\}$.
If $r=n$ then $A=X^n$ and we are done.
If $r=0$ then $A$ is a point and we are done.
Suppose $0<r<n$ and consider $a_{r+1}\in\acl(a_1,\dots,a_r)$.
By triviality, $a_{r+1}\in\acl(a_i)$ for some $i\leq r$.
Let $S=\locus(a_i,a_{r+1})\subseteq X^2$.
Then $S\to X$, under the first co-ordinate projection is a generically finite-to-one map.
By part~(a), $S$ must either be of the form $X\times\{b\}$ or $S$ is a point or $S$ is the graph of some $\sigma\in\aut X$.
In both of the first two cases $a_{r+1}=b$ is a standard point in $X$.
But as $a$ was generic in $A$ and $A$ was irreducible, this means that, after a co-ordinate permutation, $A=B\times\{b\}$ for some $B\subseteq X^{n-1}$ irreducible and complex-analytic.
The desired description of $A$ then follows by applying the induction hypothesis to $B$.
So we may assume that $S$ is the graph of some $\sigma\in \aut X$.
Again by genericity of $a$ and irreducibility of $A$, the $(r+1)$st co-ordinate of every element of $A$ is obtained by applying $\sigma$ to the $i$th co-ordinate.
That is, after a co-ordinate permutation, we get $A\subseteq f(B)$ where $B\subseteq X^{n-1}$ is the projection of $A$ to the first $n-1$ co-ordinates and $f(x_1,\dots,x_{n-1}):=(x_1,\dots,x_{n-1},\sigma(x_i))$.
As $f(B)$ is irreducible and $\dim A\geq\dim B=\dim f(B)\geq\dim A$, we get that $A=f(B)$.
The desired description of $A$ then follows from the description of $B$ given by the induction hypothesis.
\end{proof}

Lemma~\ref{scsm}(b) implies in particular that if $X$ is simply connected strongly minimal trivial, then any irreducible complex-analytic set $A\subseteq X^n$, for $n\geq 2$, is completely determined by its co-ordinate projections to $X^2$.

\begin{proposition}
\label{scsm-ctbleaut}
Suppose $X$ is a simply connected strongly minimal compact complex manifold.
Then the following are equivalent:
\begin{itemize}
\item[(i)]
$\aut X$ is countable.
\item[(ii)]
$X$ is essentially saturated and has trivial geometry.
\end{itemize}
Moreover, in this case, the language $\mathcal L_{\aut}$ consisting of a predicate symbol for the graph of each automorphism of $X$ is a full countable analytic language for $X$.
\end{proposition}

\begin{proof}
Suppose $\aut X$ is countable.
Since by Lemma~\ref{scsm}~(a) the only irreducible complex-analytic subsets of $X^2$ that project onto 
each co-ordinate are graphs of automorphisms, there can be no infinite definable family of such.
Hence by condition~(ii) of Lemma~\ref{trivial-condition} $X$ must have trivial geometry.
It follows immediately from Lemma~\ref{scsm}~(b) now that the language of automorphisms of $X$ is a full countable analytic language for $X$.
In particular, $X$ is essentially saturated.

For the converse, suppose $\aut X$ is uncountable and $X$ is essentially saturated.
Then, by the existence of a full countable language, there must exist an infinite definable family of automorphisms of $X$, which contradicts triviality.
\end{proof}

\begin{proposition}
\label{aclfix}
Suppose $X$ is a simply connected strongly minimal compact complex manifold with $\aut X$ countable.
\begin{itemize}
\item[(a)]
$(X,\mathcal{L}_{\douady})=(X,\mathcal{L}_{\aut})$ in the sense that every basic relation of one is $0$-definable in the other, and vice versa.
\item[(b)]
In $(X,\mathcal{L}_{\aut})$, and hence also in $(X,\mathcal L_{\douady})$ by part~(a),
$$\acl(\emptyset)\cap X=\bigcup_{\id\neq\sigma\in\aut X}\operatorname{Fix}(\sigma).$$
\end{itemize}
\end{proposition}

\begin{proof}
Since $\aut X$ is discrete, the graph of each automorphism is {\em isolated in $X^2$} in the sense that it lives in a zero-dimensional prime component of $D(X^2)$.
Hence the graph of each automorphism is a basic relation of $(X,\mathcal{L}_{\douady})$.
In order to prove part~(a) it therefore suffices to show that every basic relation in $(X,\mathcal L_{\douady})$ is $0$-definable in $(X,\mathcal{L}_{\aut})$.
To that end, fix $n>0$ and consider $A\subset X^n$ an irreducible complex-analytic subset.
We know by essential saturation that $A$ lives in a compact prime component of $D(X^n)$.
We wish to describe this component.

\begin{claim}
\label{douady-aut}
The prime component $C$ of $D(X^n)$ in which $A$ lives is of the form $C=X^m$ for some $0\leq m\leq n$, and the universal family restricted to $C$, $Z:=Z(X^n)|_C\subseteq C\times X^n$ is defined by equations of the form $\sigma(x_i)=x_j$ where $\sigma\in\aut X$.
\end{claim}

\begin{proof}[Proof of Claim~\ref{douady-aut}]
We prove this by induction on $n$.
For $n=1$, $A$ is either $X$ itself in which case $C=X^0$ and $Z=X$, or $A$ is a point in which case $C=X$ and $Z$ is the diagonal in $X^2$.

Suppose $n>1$.
Assume moreover that there exists a co-ordinate projection $\pi:X^n\to X$ such that $\pi(A)$ is a point.
Then after a possible permutation of co-ordinates we have that $A=A'\times\{a\}$ for some irreducible complex-analytic $A'\subseteq X^{n-1}$ and $a\in X$.
Now $A'$ lives in some prime component $C'$ of $D(X^{n-1})$ with $Z'\subseteq C'\times X^{n-1}$ the restriction of the universal family to $C'$.
As $A=A'\times\{a\}$ the family $Z\to C$ is obtained from $Z'\to C'$ by base change with respect to $C'\times X\to C'$.
That is,
$C=C'\times X$ and
$$Z=\{(c,x_0,x_0,x_1,\dots,x_{n-1}):c\in C', x_0\in X,(c,x_1,\dots,x_{n-1})\in Z'\}.$$
Applying the induction hypothesis to $C'$ and $Z'$ we see that the claim is true of $C$ and $Z$ also.

We may therefore assume that no co-ordinate projection of $X$ is a point.
By Lemma~\ref{scsm}(a), for all co-ordinate projections $\pi:X^n\to X^2$, $\pi(A)$ is the graph of an automorphism.
One consequence of this is that $A$ itself is defined by equations of the form $\sigma(x_i)=x_j$ where $\sigma\in\aut X$; indeed,  looking at the defining formulas for $A$ given in Lemma~\ref{scsm}(b) we see that these are the only possibilities.
On the other hand, by the discreteness of $\aut X$, we also get that each $\pi(A)$ lives in a zero-dimensional prime component of $D(X^2)$.
It follows by triviality that $A$ lives in a zero-dimensional prime component of $D(X^n)$ -- this is exactly Proposition~3.4 of~\cite{moosamorarutoma}.
Hence $C=X^0$ and $Z=A$.
\end{proof}

It follows from the claim that $\mathcal A_X$, the reduct of $\mathcal A$ where only the components of the universal families $Z(X^n)$ are named, is in fact one-sorted (the sort being $X$ itself) and that the basic relations in $\mathcal A_X$ are $0$-definable in $\mathcal L_{\aut}$.
By definition the same then holds for $\mathcal L_{\douady}$.
This proves part~(a).

To prove part~(b), working in $(X,\mathcal L_{\aut})$, we let $F=\displaystyle \bigcup_{\id\neq\sigma\in\aut X}\operatorname{Fix}(\sigma)$ and show that there is a unique $1$-type in $X\setminus F$ over $F$.
Note that each member of $\aut(X)$ fixes $F$ setwise: if $f\in\operatorname{Fix}(\sigma)$ then $\tau(f)\in\operatorname{Fix}(\tau\sigma\tau^{-1})$.
Hence $\aut X$ acts on $X\setminus F$, and this action is clearly free.
It follows that $X$ is the disjoint union of $\aut X$ orbits and $F$, and each such orbit is being acted on regularly by $\aut(X)$.
It is then clear that for any $a,b\in X\setminus F$ there is an automorphism of the {\em structure} $(X,\mathcal L_{\aut})$ which fixes $F$ pointwise and takes $a$ to $b$.
So all elements of $X\setminus F$ have the same $\mathcal L_{\aut}$-type over the empty set. 
\end{proof}

Now we investigate $\aleph_0$-categoricity for simply connected strongly minimal compact complex manifolds.

\begin{theorem}
\label{scsm-finiteaut}
Suppose $X$ is a simply connected strongly minimal compact complex manifold.
Then the following are equivalent.
\begin{itemize}
\item[(i)]
$\aut X$ is finite.
\item[(ii)]
$X$ is essentially saturated and $(X,\mathcal L_{\aut})$ is $\aleph_0$-categorical.
\item[(iii)]
$X$ is essentially saturated and $(X,\mathcal L_{\douady})$ is $\aleph_0$-categorical.
\item[(iv)]
$X$ is essentially $\aleph_0$-categorical.
\end{itemize}
\end{theorem}

\begin{proof}
(i) $\implies$ (ii).
Assume that $\aut X$ is finite.
By Proposition~\ref{scsm-ctbleaut} we know that $X$ is essentially saturated and trivial and that $\mathcal L_{\aut}$ is a full countable analytic language for $X$.
We need to check that for each $n>0$, there are only finitely many $0$-definable subsets of $X^n$ in $(X,\mathcal{L}_{\aut})$.
By Lemma~\ref{qefcl}, since we are working in an analytic language, it suffices to show that there are only finitely many irreducible complex-analytic subsets of $X^n$ that are $\acl(\emptyset)$-definable in $(X,\mathcal{L}_{\aut})$.

For $n=1$ we need to count the number of $\acl(\emptyset)$-definable points in $X$.
Since $\aut X$ is finite and the set of fixed points of each nontrivial member of $\aut X$ is finite, Proposition~\ref{aclfix}(b) tells us that $\acl(\emptyset)\cap X$ is finite.
The $n=2$ case is taken care of by the description of the irreducible complex-analytic subsets of $X^2$ given by Lemma~\ref{scsm}(a): the only possibilities are $X^2$, points, slices, or graphs of automorphisms.
The first and last of these only contribute finitely many.
In the case of points or slices, note that the singletons involved, by $\acl(\emptyset)$-definability and saturation, must be in $\acl(\emptyset)\cap X$ -- and hence these also only contribute finitely many possibilities for irreducible complex-analytic subsets of $X^2$.

The $n=2$ case now implies the general case:
If $n>2$ and $A\subseteq X^n$ is irreducible complex-analytic $\acl(\emptyset)$-definable, then each of the projections of $A$ to $X^2$ are also irreducible complex-analytic $\acl(\emptyset)$-definable.
But by triviality $A$ is determined by its projections to $X^2$ -- see Lemma~\ref{scsm}(b).
Hence there are only finitely many possibilities for $A$.

(ii) $\implies$ (iii).
Because of $\aleph_0$-categoricity $X$ must have trivial geometry (see Fact~\ref{cat-triv}).
The implication now follows immediately from Proposition~\ref{aclfix}(a).

(iii) $\implies$ (iv).
Clear.

(iv) $\implies$ (i).
Suppose $\mathcal L$ is a full countable language for $X$ such that $(X,\mathcal{L})$ is $\aleph_0$-categorical.
Then $(X,\mathcal L)$ must have trivial geometry.
It follows that every automorphism of $X$ is $\acl(\emptyset)$-definable in $(X,\mathcal{L})$.
By $\aleph_0$-categoricity, there can be only finitely many such.
\end{proof}

This, together with an example of McMullen discussed below, resolves in the negative a conjecture of the second author and Thomas Scanlon from~\cite{pillayscanlon2001}:

\begin{corollary}
There exist trivial strongly minimal compact K\"ahler manifolds which are not $\aleph_0$-categorical in any full countable language.
\end{corollary}

\begin{proof}
This comes from the study of generic analytic K3 surfaces due to Gross, McMullen and Oguiso, though we were informed by the survey article~\cite{macristellari}.
An {\em analytic K3 surface} is a smooth simply connected compact surface $X$ with trivial canonical bundle.
They are K\"ahler manifolds (and hence essentially saturated).
A K3 surface is {\em generic} if it has trivial Picard group (these are in fact dense in the moduli space of K3 surfaces).
Generic K3 surfaces are strongly minimal (as any curve on $X$ would give rise to an effective divisor and hence a nontrivial line bundle).
Oguiso has shown that a generic K3 surface either has trivial automorphism group or $\aut X=\mathbb Z$.
In particular, by Proposition~\ref{scsm-ctbleaut}, all generic K3 surfaces have trivial geometry.
McMullen produced examples with $\aut X=\mathbb Z$, which by Theorem~\ref{scsm-finiteaut} cannot be $\aleph_0$-categorical in any full countable language.
\end{proof}

\section{The general case}

\noindent
Fix a strongly minimal compact complex manifold $X$.
There is no harm in assuming that $\dim(X) > 1$, as we are interested in $X$ with trivial geometry.
Some of what we did in the previous section goes through without the assumption of simply connectedness if we replace automorphisms by finite-to-finite correspondences, but there are additional complications.
By a {\em finite-to-finite coorrespondence on $X$} we mean an irreducible complex-analytic subset $S\subset X^2$ such that both co-ordinate projections are surjective finite-to-one maps.
We denote the set of all finite-to-finite correspondences by $\corr X$.
In the simply connected case we used normalisations to see that any such correspondence must be the graph of an automorphism.
In the general setting that argument yields only the following:

\begin{lemma}
\label{mto1}
Suppose $\dim X>1$, $S\in\corr X$, and $\pi:S\to X$ is one of the two co-ordinate projections.
Then $\pi$ is everywhere $m$-to-one for some $m>0$.
\end{lemma}

\begin{proof}
We know that $\pi$ is finite-to-one and hence $m$-to-one generically for some $m>0$.
Note that $m$ is the least cardinality of any of the fibres of $\pi$.
Let $\rho:S'\to S$ be a normalisation of $S$.
Then $\rho$ is generically one-to-one so that $\pi\circ\rho:S'\to X$ is also generically $m$-to-one.
But by Fact~\ref{norm}, $\pi\circ\rho$ is an unramified covering and hence everywhere $m$-to-1.
It follows that $\rho$ is bijective and $\pi$ is everywhere $m$-to-one.
\end{proof}

Unlike in the simply connected case, triviality will not completely reduce the study of irreducible complex-analytic subsets of $X^n$ to members of $\corr X$.
The main problem is that intersections of pull-backs of correspondences may not be irreducible, and so their irreducible components need to be taken into account.

\begin{definition}
By a {\em generalised correspondence} on $X$ we mean an irreducible $\dim(X)$-dimensional complex-analytic subset of $X^n$ that projects onto $X$ in each co-ordinate.
\end{definition}

\begin{lemma}
\label{gencorr=comp}
Suppose $A\subseteq X^n$ is a generalised correspondence on a strongly minimal compact complex manifold.
For each $i=2,\dots n$, let $\pi_i:X^n\to X^2$ be the co-ordinate projection $(x_1,\dots,x_n)\mapsto (x_1,x_i)$.
Then $S_i:=\pi_i(A)\in\corr X$ for all $i=1,\dots,\ell$ and $A$ is an irreducible component of $\bigcap_{i=2}^n\pi_i^{-1}(S_i)$.
\end{lemma}

\begin{proof}
Clearly $S_i\subseteq X^2$ is irreducible complex-analytic and projects onto $X$ in each co-ordinate.
Hence $\dim X\leq\dim(S_i)\leq\dim A=\dim X$, and so $\dim(S_i)=\dim X$ also.
From this it follows that the co-ordinate projections of $S_i$ onto $X$ are finite-to-one.
That is, $S_i\in\corr X$.

Note that $B:=\bigcap_{i=2}^n\pi_i^{-1}(S_i)$ is of $\acl$-dimension at most $1$ as every co-ordinate is algebraic over the first co-ordinate.
Hence $\dim B\leq \dim X$.
But $A\subseteq B$, so that $\dim B=\dim X$, and $A$ is an irreducible component of $B$.
\end{proof}

We have the following analogue of Lemma~\ref{scsm}.

\begin{lemma}
\label{sm}
Suppose $X$ is a strongly minimal compact complex manifold.
\begin{itemize}
\item[(a)]
The only irreducible complex-analytic subsets of $X^2$ are points, $X^2$ itself, vertical and horizontal ``slices'' $\{a\}\times X$ and $X\times \{a\}$ where $a\in X$, and finite-to-finite correspondences.
\item[(b)]
Suppose moreover that $X$ has trivial geometry.
If $A\subseteq X^n$ is an irreducible complex-analytic subset, then, after some permutation of the co-ordinates, $A$ is a product of generalised correspondences and singletons.
\end{itemize}
\end{lemma}

\begin{proof}
Suppose
$A\subset X^2$ is a proper irreducible complex-analytic subset that projects onto both co-ordinates.
By strong minimality all the fibres of the projections are either finite or all of $X$.
By dimension calculations the projections must be generically finite-to-one.
But then $\dim A=\dim X$, and so by irreducibility, all the fibres of both projections must be finite.
That is, $A\in\corr X$.

For part~(b) the $n=1$ case is clear, and the $n=2$ case is part~(a).
For $n>2$, let $a=(a_1,\dots,a_n)\in X(\mathcal{A}')$ be a generic point of $A$ in a sufficiently saturated elementary extension $\mathcal A'$ of $\mathcal A$.
Note that if some $a_i\in\acl(\emptyset)$ then $a_i$ is a standard point of $X$ and so after permuting co-ordinates $A$ is of the form $A'\times\{a_i\}$, and we are done by induction.
We may therefore assume that all $a_i\notin\acl(\emptyset)$.
Let $\{b_1,\dots,b_r\}$ be an $\acl$-basis for $\{a_1,\dots,a_n\}$.
By triviality, and the fact that no co-ordinate is in $\acl(\emptyset)$, each $a_j$ is in $\acl(b_i)$ for a unique $b_i$.
Hence, after permuting the co-ordinates, we can write
$(a_1,\dots,a_n)=(\bar b_1,\dots,\bar b_r)$ where $\bar b_i=(b_i=b_{i,1},\dots,b_{i,k_i})$ and $b_{i,j}\in\acl(b_{i,1})$, for all $i\leq r$ and $j\leq k_i$.
For each $i\leq r$ let $A_i=\locus(\bar b_i)$.
Then $A_i\subseteq X^{k_i}$ is irreducible complex-analytic; it is of dimension $\dim X$ since $\bar b_i$ is of $\acl$-dimension $1$, and it projects onto $X$ in each co-ordinate since every co-ordinate of $\bar b_i$ is not in  $\acl(\emptyset)$.
That is, each $A_i$ is a generalised correspondence.
Since $\{b_{1,1},\dots,b_{r,1}\}$ is $\acl$-independent, we have $A=A_1\times\cdots\times A_r$, as desired.
\end{proof}

\begin{definition}
\label{lcorr}
Given a strongly minimal compact complex manifold $X$, the {\em language of generalised correspondences} for $X$, denoted by $\mathcal{L}_{\gc}$, is the language where there is a predicate for each generalised correspondence on $X$.
\end{definition}

It follows from Lemma~\ref{gencorr=comp} that if $\corr X$ is countable then so is $\mathcal L_{\gc}$.

The arguments for Proposition~\ref{scsm-ctbleaut} and Proposition~\ref{aclfix}(a) now generalise to:

\begin{proposition}
\label{sm-ctblecorr}
Suppose $X$ is a strongly minimal compact complex manifold.
Then the following are equivalent:
\begin{itemize}
\item[(i)]
$\corr X$ is countable.
\item[(ii)]
$X$ is essentially saturated and has trivial geometry.
\end{itemize}
In this case $\mathcal{L}_{\gc}$ is a full countable analytic language for $X$, and $(X,\mathcal{L}_{\douady})=(X,\mathcal{L}_{\gc})$ in the sense that every basic relation of one is $0$-definable in the other, and vice versa.
\end{proposition}

\begin{proof}
The equivalence of (i) and~(ii) is almost exactly as in Proposition~\ref{scsm-ctbleaut}.
 The countability of $\corr X$ implies the countability of $\mathcal L_{\gc}$ by Lemma~\ref{gencorr=comp}.
It also implies that there is no infinite definable family of finite-to-finite correspondences.
Since  finite-to-finite correspondences are the only irreducible complex-analytic subsets of $X^2$ that project onto both co-ordinates (Lemma~\ref{sm}(a)), we see by condition~(ii) of Lemma~\ref{trivial-condition} that the geometry on $X$ must be trivial.
Hence every irreducible complex-analytic subset is up to a co-ordinate permutation a product of generalised correspondences and singletons (Lemma~\ref{sm}(b)).
From this it follows that $\mathcal{L}_{\gc}$ is a full countable analytic language for $X$.
For the converse, note that if $X$ were essentially saturated and $\corr X$ were uncountable then there would be an infinite definable family of finite-to-finite correspondences, contradicting triviality.

Now suppose that the equivalent conditions~(i) and~(ii) are satisfied.
We want to show that $(X,\mathcal{L}_{\douady})=(X,\mathcal{L}_{\gc})$.
Following the argument for automorphisms, in order to show that every generalised correspondence is $0$-definable in $(X,\mathcal{L}_{\douady})$ we prove that they are isolated in the sense that each one lives in a zero-dimensional prime component of the Douady space.
Suppose $A\subseteq X^n$ is a generalised correspondence on $X$ living in the irreducible prime component $C$ of the Douady space of $X^n$.
There exists a proper complex-analytic subset $E\subset C$, such that distinct points in $C\setminus E$ correspond to distinct irreducible $\dim(X)$-dimensional complex-analytic subsets of $X^n$.
In Lemma~2.3 of~\cite{moosamorarutoma} it is pointed out that projecting onto a singleton in some co-ordinate is a property that is preserved in irreducible components of $D(X^n)$.
Since $A$ does not project onto a singleton in any co-ordinate, this must also be true of each of the complex-analytic subsets of $X^n$ given by points in $C$.
By strong  minimality it follows that each of the irreducible complex-analytic sets corresponding to points of $C\setminus E$ project onto $X$ in every co-ordinate.
That is, distinct points of $C\setminus E$ give rise to distinct generalised correspondences on $X$.
As there are only countably many generalised correspondences, $C$ must be zero-dimensional, as desired.

Finally, still assuming countability of $\corr X$, we need to prove that every basic relation in $(X,\mathcal{L}_{\douady})$ is $0$-definable in $(X,\mathcal{L}_{\gc})$.
Here the argument is exactly as in Proposition~\ref{aclfix}(a) once we replace Claim~\ref{douady-aut} by:
{\em For every irreducible complex-analytic $A\subseteq X^n$,
the prime component $C$ of $D(X^n)$ in which $A$ lives is of the form $C=X^m$ for some $0\leq m\leq n$, and the universal family restricted to $C$, $Z:=Z(X^n)|_C\subseteq C\times X^n$ is defined by a conjunction of atomic $\mathcal{L}_{\gc}$-formulas.}
The claim is also proved by induction on $n$, with $n=1$ being clear.
For $n>1$, if some co-ordinate projection of $A$ to $X$ is a singleton then one reduces to the induction hypothesis exactly as in the proof of Claim~\ref{douady-aut}.
Hence, by Lemma~\ref{sm}(a), we may assume that every co-ordinate projection of $A$ to $X^2$ is a finite-to-finite correspondence on $X$.
We have already seen that the elements of $\corr X$ are all isolated in $X^2$, so all the projections of $A$ to $X^2$ are isolated.
It follows by triviality, using Proposition~3.4 of~\cite{moosamorarutoma}, that $A$ must be isolated in $X^n$.
Hence $C=X^0$ and $Z=A$.
So it remains to observe that in this case, $A$ itself is defined by a conjunction of atomic $\mathcal{L}_{\gc}$-formulas.
But that is just what Lemma~\ref{sm}(b) says, given that no co-ordinate projection of $A$ is a singleton.
\end{proof}

We now need to analyse the structure $(X,\mathcal L_{\gc})$.
In particular we need to describe the algebraic closure of the empty set.
Unlike in the simply connected case, we are not just working with a pure group action.
However, as it turns out, we are not so very far away from that situation.

\begin{definition}
Given $S,T\in\corr X$ and $x\in X$ we set
\begin{itemize}
\item[]
$T\circ S:=\big\{(a,b)\in X^2 : \text{ for some $c\in X$, $(a,c)\in S$ and $(c,b)\in T$}\big\}$,
\item[]
$S^{-1}:=\big\{(a,b)\in X^2:(b,a)\in S\big\}$,
\item[]
$\Delta:=\big\{(a,a):a\in X\big\}$
\item[]
$\orb(x):=\big\{a\in X:(x,a)\in U \text{ for some $U\in\corr X$}\big\}$,
\item[]
$X_{00}:=\big\{a\in X:(a,a)\in U\text{ for some $U\in\corr X$ with $U\neq\Delta$}\big\}$,
\item[]
$\displaystyle X_0:=\bigcup_{a\in X_{00}}\orb(a)$.
\end{itemize}
\end{definition}

Note that in the simply connected case when the only finite-to-finite correspondences are the graphs of automorphisms, $T\circ S$ is the graph of the composition of the automorphisms, $S^{-1}$ is the graph of the inverse automorphism, $\Delta$ is the graph of the identity automorphism, $\orb(x)$ is the orbit of $x$ under the action of $\aut X$ on $X$, and $X_0=X_{00}$ is the set of fixed points of $X$ under this action.
We will eventually prove that even when $X$ is not simply connected, $X_0$ is the algebraic closure of the empty set in $(X,\mathcal L_{\gc})$.

\begin{proposition}
\label{orbequiv}
The relation $y\in\orb(x)$ is an equivalence relation on $X$.
\end{proposition}

\begin{proof}
Reflexivity is by the fact that $\Delta\in\corr X$ and symmetry is by the fact that if $S\in \corr X$ then $S^{-1}\in \corr X$.
Finally, while it is not necessarily the case that $T\circ S\in\corr X$ whenever $S,T\in\corr X$, the following lemma shows that the irreducible components of $T\circ S$ are -- and that suffices for transitivity.
\end{proof}

\begin{lemma}
\label{composedgencorr}
Suppose $n>1$ and $A_1,A_2\subseteq X^n$ are generalised correspondences such that $\pi(A_1)=\pi(A_2)$, where $\pi:X^n\to X^{n-1}$ is the projection onto the last $n-1$ co-ordinates.
Let
$$B:=\big\{(a_1,a_2)\in X^2 : \text{ for some $z\in X^{n-1}$, $(a_1,z)\in A_1$ and $(a_2,z)\in A_2$}\big\}$$
Then every irreducible component of $B$ is in $\corr X$.
Moreover, unless $A_1=A_2$, none of these components is $\Delta$.

In particular, if $S,T\in\corr X$, then every irreducible component of $T\circ S$ is in $\corr X$, and unless $S=T^{-1}$ none of these components is $\Delta$.
\end{lemma}

\begin{proof}
The ``in particular'' clause follows by letting $n=2$, $A_1=S$, and $A_2=T^{-1}$.

For the ``moreover'' clause, observe that if  $\Delta\subseteq B$, then $A_1\cap A_2$ projects onto $X$ in the first co-ordinate, so that $\dim(A_1\cap A_2)=\dim X$, and hence $A_1=A_2$.

Now let us prove that main statement.
It is not hard to see that $B\subseteq X^2$ is a complex-analytic set that projects onto $X$ in each co-ordinate in a finite-to-one manner.
It follows from Lemma~\ref{sm}(a) that each irreducible component of $B$ must be either a singleton or a finite-to-finite correspondence.
We need to rule out the possibility of zero-dimensional components.

Let $A=\pi(A_1)=\pi(A_2)\subseteq X^{n-1}$, and let $\rho:A'\to A$ be a normalisation of $A$.
Note that $A$ is again a generalised correspondence, and in particular projects onto $X$ in a finite-to-one manner in each co-ordinate.
Then, by Fact~\ref{norm}, composing $\rho$ with any co-ordinate projection shows that $A'$ is a unramified cover of $X$.
In particular, $A'$ is smooth and of dimension $\dim X$.

Let $p_1:X\times A'\times X\to X^n$ be the holomorphic surjection $(x_1,z,x_2)\mapsto \big(x_1,\rho(z)\big)$, let $p_2:X\times A'\times X\to X^n$ be $(x_1,z,x_2)\mapsto \big(x_2,\rho(z)\big)$
and let
$q:X\times A'\times X\to X^2$ be $(x_1,z,x_2)\mapsto (x_1,x_2)$.
Set $W:=p_1^{-1}(A_1)\cap p_2^{-1}(A_2)\subseteq X\times A'\times X$.
By the surjectivity of $\rho$, $B=q(W)$.
By the smoothness of $X\times A'\times X$ we know that each irreducible component of $W$ is of dimension at least $2\dim X+2\dim X-3\dim X=\dim X$.
But since $q|_W$ is also finite-to-one, every irreducible component of $B$ is at least of dimension $\dim X$, as desired.
\end{proof}

Note that the intersection of distinct generalised correspondences, $A_1,A_2\subseteq X^n$, must be finite.
Indeed, as the first co-ordinate projection of $A_1$ is finite-to-one onto $X$, if $A_1\cap A_2$ were infinite, then its projection, being complex-analytic, would be all of $X$ by strong minimality.
Hence $\dim(A_1\cap A_2)\geq\dim X$.
Irreducibility would then imply that $A_1=A_2$.
So such intersections will always land in the algebraic closure of the empty set in $(X,\mathcal L_{\gc})$.
Hence the following lemma is a necessary part of showing that that algebraic closure is $X_0$.

\begin{lemma}
\label{gencorrint}
Suppose $A_1,A_2\subseteq X^n$ are distinct generalised correspondences.
Then $A_1\cap A_2\subseteq X_0^n$.
\end{lemma}

\begin{proof}
By induction on $n>0$.
The case of $n=1$ is vacuous since then $A_i=X$ for $i=1,2$.
Suppose $n>1$.
Let $\pi:X^n\to X^{n-1}$ be some co-ordinate projection.
Note that $\pi(A_1)$ and $\pi(A_2)$ are generalised correspondences, so that by induction if $\pi(A_1)\neq\pi(A_2)$ then $\pi(A_1)\cap\pi(A_2)\subseteq X_0^{n-1}$.
If this were the case for all co-ordinate projections to $X^{n-1}$, then we would be done. So we may assume that for some $\pi$, $\pi(A_1)=\pi(A_2)$.
After permuting co-ordinates we may in addition assume that $\pi$ is the projection onto the last $n-1$ co-ordinates.
We are thus in a situation to which Lemma~\ref{composedgencorr} applies.
Now let $(a,z)\in A_1\cap A_2$, where $a\in X$ and $z=(z_2,\dots,z_n)\in X^{n-1}$.
Then $(a,a)$ is in the set $B$ of Lemma~\ref{composedgencorr}, and so by that lemma, $(a,a)\in S$ for some $\Delta\neq S\in\corr X$.
Hence $a\in X_{00}$.
By Lemma~\ref{gencorr=comp}, for all $i=2,\dots,n$, $(a,z_i)\in T$ for some $T\in \corr X$.
It follows that each $z_i\in\orb(a)$, and so $z_i\in X_0$, as desired.
\end{proof}

The following is the key step in proving that any two elements outside $X_0$ have the same $\mathcal L_{\gc}$-type.

\begin{proposition}
\label{orbperm}
Suppose $\corr X$ is countable.
For all $a,b\in X\setminus X_0$, there is a bijection $\phi:\orb(a)\to\orb(b)$ such that $\phi(a)=b$, and for any generalised correspondence $A\subseteq X^n$ and $x_1,\dots,x_n\in\orb(a)$,
$(x_1,\dots,x_n)\in A$ if and only if $\big(\phi(x_1),\dots,\phi(x_n)\big)\in A$.
\end{proposition}

\begin{proof}
Let $(S_i:i<\omega)$ be an enumeration of $\corr X\setminus\{\Delta\}$.

Let $\mathcal A'$ be a sufficiently saturated elementary extension of $\mathcal A$,and denote by $X(\mathcal A')$ the interpretation of $X$ in the extension.
Fix a generic point of $X$ in $\mathcal A'$ -- that is a point $c\in X(\mathcal A')\setminus X$.
Let
$$\orb(c):=\{x\in X(\mathcal A'):(x,c)\in S(\mathcal A')\text{ for some }S\in\corr X\}$$
We will show that there is a bijection $\phi:\orb(a)\to\orb(c)$ such that $\phi(a)=c$, and for any generalised correspondence $A\subseteq X^n$ and $x_1,\dots,x_n\in\orb(a)$,
$(x_1,\dots,x_n)\in A$ if and only if $\big(\phi(x_1),\dots,\phi(x_n)\big)\in A(\mathcal A')$.
Applying this to $b$ also, will prove the proposition.

Let $(c_j:1\leq j<\omega)$ be an enumeration of $\displaystyle\bigcup_{i<\omega} S_i(\mathcal A')_c$, where $S_i(\mathcal A')_c:=\{x\in X(\mathcal A'):(x,c)\in S_i(\mathcal A')\}$.
Moreover, assume the enumeration is coherent in the sense that for some increasing sequence $\ell_0<\ell_1<\cdots$, $(c_1,\dots,c_{\ell_r})$ is an enumeration of $\displaystyle\bigcup_{i\leq r} S_i(\mathcal A')_c$.
For each $r<\omega$ let $A_r:=\locus(c,c_1,\dots,c_{\ell_r})\subseteq X^{\ell_r+1}$.

\begin{claim}
\label{arcorr}
Each $A_r$ is a generalised correspondence.
\end{claim}
\begin{proof}[Proof of~\ref{arcorr}]
That $A_r$ is irreducible and complex-analytic is by definition.
Note that for each $j$, $c\in\acl(c_j)$ since $(c_j,c)\in S(\mathcal A')$ for some $S\in\corr X$.
Hence the genericity of $c$ implies the genericity of each $c_j$.
It follows that the co-ordinate projections of $A_r$ are all onto $X$.
Finally, it is $\dim(X)$-dimensional because each $c_j\in\acl(c)$ and so the first co-ordinate projection is generically finite-to-one.
\end{proof}

The coherence of the enumeration yields a direct system of surjective maps $X\leftarrow A_0\leftarrow A_1\leftarrow\cdots$ given by the natural initial segment co-ordinate projections.
We can therefore find $(a_j:1\leq j<\omega)$ such that for all $r$, $(a,a_1,\dots,a_{\ell_r})\in A_r$.

\begin{claim}
\label{enumorb}
$(a_j:1\leq j<\omega)$ enumerates $\orb(a)\setminus\{a\}$.
\end{claim}

\begin{proof}[Proof of~\ref{enumorb}]
Note that $\orb(a)\setminus\{a\}$ is the increasing union of the sets $\displaystyle \bigcup_{i=0}^r(S_i)_a$, as $r$ goes to infinity. (Note that $a$ is not in any of the latter as $a\notin X_0$.)
Hence it suffices to show, fixing $r$, that $(a_1,\dots,a_{\ell_r})$ enumerates $\displaystyle \bigcup_{i=0}^r(S_i)_a$.

By construction, for each $j\leq\ell_r$, $(c_j,c)\in S_i(\mathcal A')$ for some $i\leq r$.
By genericity of $c$, it follows that the projection to the $(1+j,1)$-co-ordinate of $A_r=\locus(c,c_1,\dots,c_{\ell_r})$ has infinite intersection with $S_i$.
But as $A_r$ is a generalised correspondence, this projection is itself a finite-to-finite correspondence (see Lemma~\ref{gencorr=comp}), so that it must equal $S_i$.
So $(a_j,a)\in S_i$.
It suffices to show therefore that $\ell_r$ is the cardinality of $\displaystyle \bigcup_{i=0}^r(S_i)_a$.
By construction $\ell_r$ is the cardinality of $\displaystyle \bigcup_{i=0}^rS_i(\mathcal A')_c$.
But note that for $i\neq j$, $S_i\cap S_j$ is finite and so $S_i(\mathcal A')_c\cap S_j(\mathcal A')_c=\emptyset$.
Hence,
\begin{eqnarray*}
\ell_r
&=&
\big|\bigcup_{i=0}^rS_i(\mathcal A')_c\big|\\
&=&
\sum_{i=0}^r|S_i(\mathcal A')_c|\\
&=&
\sum_{i=0}^rm_i
\end{eqnarray*}
where $m_i$ is the cardinality of the fibres of the second co-ordinate projection on $S_i$ -- which by Lemma~\ref{mto1} is a constant.
On the other hand, for $i\neq j$, $S_i\cap S_j\subseteq X_0^2$ by Lemma~\ref{gencorrint}, and so as $a\notin X_0$, $(S_i)_a\cap(S_j)_a=\emptyset$ also.
Hence
$$\big|\bigcup_{i=0}^r(S_i)_a\big|=\sum_{i=0}^r|(S_i)_a|=\sum_{i=0}^rm_i$$
as well.
\end{proof}

Since $(c_0:=c,c_j:1\leq j<\omega)$ enumerates $\orb(c)$ by construction, and now we know that $(a_0:=a,a_j:1\leq j<\omega)$ enumerates $\orb(a)$, all that remains to be proved is that if $A\subseteq X^n$ is any generalised correspondence, and $i_1,\dots,i_n<\omega$ are arbitrary, then $(a_{i_1},\dots,a_{i_n})\in A$ if and only if $(c_{i_1},\dots,c_{i_n})\in A(\mathcal A')$.
Fix $r$ so that $i_1,\dots,i_n$ are all $\leq \ell_r$.
For the right-to-left direction, note that $(c_{i_1},\dots,c_{i_n})\in A(\mathcal A')$ implies that the $(i_1,\dots,i_n)$-co-ordinate projection of $A_r$ has infinite intersection with $A$, and so as both are generalised correspondences, must equal $A$.
Hence, $(a_{i_1},\dots,a_{i_n})\in A$.
Conversely, suppose $(c_{i_1},\dots,c_{i_n})\notin A(\mathcal A')$, $(a_{i_1},\dots,a_{i_n})\in A$, and seek a contradiction.
Then the $(i_1,\dots,i_n)$-co-ordinate projection of $A_r$, say $B\subseteq X^n$, is a generalised correspondence, different from $A$.
By Lemma~\ref{gencorrint}, $A\cap B\subseteq X_0^n$, so that $a_{i_1}\in X_0$.
Hence $a\in\orb(a_{i_1})\subseteq X_0$.
This contradiction proves that $(a_{i_1},\dots,a_{i_n})\notin A$, and thus completes the proof of Proposition~\ref{orbperm}.
\end{proof}

Putting~\ref{orbequiv} and~\ref{orbperm} together, we obtain the desired characterisation of algebraic closure in $(X,\mathcal L_{\gc})$.

\begin{proposition}
\label{aclx0}
If $\corr X$ is countable, then $\acl(\emptyset)\cap X=X_0$ in $(X,\mathcal L_{\gc})$.
\end{proposition}

\begin{proof}
For the right-to-left containment, suppose $a\in X_0$.
Then by definition $a\in\orb(b)$ for some $b\in X_{00}$.
But then $(b,b)\in S\cap\Delta$ for some $\Delta\neq S\in\corr X$.
As this intersection must be finite, $b\in\acl(\emptyset)$.
Now $a\in\orb(b)$ implies that $a\in\acl(b)\subseteq\acl(\emptyset)$, as desired.

To prove the left-to-right direction we take $a,b\in X\setminus X_0$ and show that they have the same type in $(X,\mathcal L_{\gc})$.
We do this by exhibiting an automorphism of $(X,\mathcal L_{\gc})$ taking $a$ to $b$.
By Proposition~\ref{orbequiv}, $X$ is partitioned into disjoint orbits.
Let $\phi:\orb(a)\to\orb(b)$ be the bijection given by Proposition~\ref{orbperm}.
We define $\sigma:X\to X$ to be the permutation that is the identity on every orbit except $\orb(a)$ and $\orb(b)$, $\sigma|_{\orb(a)}=\phi$, and $\sigma|_{\orb(b)}=\phi^{-1}$.
To show that $\sigma$ is an $\mathcal L_{\gc}$-automorphism we need to show that it preserves all the generalised correspondences.
Let $A\subseteq X^n$ be a generalised correspondence, and $x=(x_1,\dots,x_n)\in A$.
Then all the $x_i$'s are in the same orbit.
If that orbit is $\orb(a)$ or $\orb(b)$ then Proposition~\ref{orbperm} implies that $\sigma(x)\in A$.
If not, then $\sigma(x)=x\in A$.
\end{proof}

Now the characterisation of essential $\aleph_0$-categoricity goes through:

\begin{theorem}
\label{sm-finitecorr}
Suppose $X$ is a strongly minimal compact complex manifold.
Then the following are equivalent.
\begin{itemize}
\item[(i)]
$\corr X$ is finite.
\item[(ii)]
$X$ is essentially saturated and $(X,\mathcal{L}_{\gc})$ is $\aleph_0$-categorical.
\item[(iii)]
$X$ is essentially saturated and $(X,\mathcal L_{\douady})$ is $\aleph_0$-categorical.
\item[(iv)]
$X$ is essentially $\aleph_0$-categorical.
\end{itemize}
\end{theorem}

\begin{proof}
The proof is analogous the simply connected case (Theorem~\ref{scsm-finiteaut}).

(i)$\implies$(ii).
Assume $\corr X$ is finite.
First observe that $\acl(\emptyset)\cap X$ is finite.
Indeed, by Proposition~\ref{aclx0}, $\displaystyle \acl(\emptyset)\cap X=X_0=\bigcup_{a\in X_{00}}\orb(a)$.
Since $\corr X$ is finite and the intersection of distinct finite-to-finite correspondences are finite, $X_{00}$ is finite.
But every orbit is also finite.
Hence $X_0$ is finite, as desired.

Also, by Lemma~\ref{gencorr=comp} there are only finitely many generalised correspondences in $X^n$, for each $n>0$.

Now, if $A\subseteq X^n$ is irreducible complex-analytic, then (after a permutation of co-ordinates) $A$ is a product of generalised correspondences and singletons -- this is Lemma~\ref{sm}(b).
By saturation of $(X,\mathcal L_{\gc})$ -- which holds because of Proposition~\ref{sm-ctblecorr} -- one can use automorphisms to show that if $A$ is $\acl(\emptyset)$-definable in $(X,\mathcal L_{\gc})$, then the singletons that appear must come from $\acl(\emptyset)\cap X$.
Hence there are only finitely many $\acl(\emptyset)$-definable irreducible complex-analytic subsets of $X^n$ in $(X,\mathcal L_{\gc})$, for all $n>0$.
Exactly as in Theorem~\ref{scsm-finiteaut}, this implies $\aleph_0$-categoricity.

(ii)$\implies$(iii).
This is because $\aleph_0$-categoricity implies triviality, and then we know that $(X,\mathcal L_{\douady})$ is inter-$0$-definable with $(X,\mathcal{L}_{\gc})$ by Proposition~\ref{sm-ctblecorr}.

(iii)$\implies$(iv). Clear.

(iv)$\implies$(i).
Let $\mathcal L$ be a full countable language for $X$ such that $(X,\mathcal L)$ is $\aleph_0$-categorical.
Again we conclude that $X$ has trivial geometry and hence every finite-to-finite correspondence is $\acl(\emptyset)$-definable in $(X,\mathcal L)$.
So, by $\aleph_0$-categoricity, $\corr X$ is finite.
\end{proof}

In conclusion let us discuss some possible extensions and generalisations.
First of all, we expect the equivalence of~(iii) and ~(iv) -- that is, the robustness of $\aleph_0$-categoricity -- to hold for arbitrary compact complex varieties, and not just for smooth strongly minimal ones.
In fact, the right setting in which to investigate $\aleph_0$-categoricity in bimeromorphic geometry seems to be what we might call {\em meromorphic varieties}: Zariski open subsets of compact complex varieties.
On the other hand, it also makes sense to leave the complex-analytic context altogether and ask whether~(i), (ii), and~(iv) are equivalent for strongly minimal {\em complete pre-smooth Zariski-type structures} in the sense of Zilber (see~\cite{zilberbook}).
In fact, much of what we have done here works in that setting; the only parts of our argument that do not immediately extend are Lemmas~\ref{mto1} and~\ref{composedgencorr} where we make essential use of the existence of normalisations and the purity of branch theorem.

%\vfill
%\pagebreak

%\bibliographystyle{plain}
%\bibliography{../ccs}

\end{document}